\documentclass[12pt]{article}
\usepackage{amsmath,amsfonts,amssymb,amsthm}
\usepackage{xcolor}
\usepackage{url} 
\newtheorem{proposition}{Proposition}[section]
\newtheorem{theorem}[proposition]{Theorem}
\newtheorem{corollary}[proposition]{Corollary}
\newtheorem{lemma}[proposition]{Lemma}
\newtheorem{remark}[proposition]{Remark}

\newtheorem{example}[proposition]{Example}

\newtheorem{algorithm}[proposition]{Algorithm}

\newcommand{\nc}{\newcommand}
\nc{\I}{{\bf 1}}
\nc{\bG}{{G}}
\nc{\bS}{{\mathbf S}}
\nc{\bN}{{\mathbf N}}
\nc{\bM}{{\mathbf M}}
\nc{\cB}{{\mathcal B}}
\nc{\cG}{{\mathcal G}}
\nc{\cS}{{\mathcal S}}
\nc{\cM}{{\mathcal M}}
\nc{\R}{{\mathbb R}}
\nc{\N}{{\mathbb N}}
\nc{\Z}{{\mathbb Z}}

\nc{\BP}{\mathbb{P}}
\nc{\BE}{\mathbb{E}}
\nc{\BQ}{\mathbb{Q}}

\numberwithin{equation}{section}

\begin{document}

\renewcommand{\thefootnote}{\fnsymbol{footnote}}
\author{{\sc G\"unter Last\thanks{Karlsruhe Institute of Technology,
    Germany. E-mail: \texttt{guenter.last@kit.edu}}
and Hermann Thorisson\thanks{University of Iceland, Iceland. 
E-mail: \texttt{hermann@hi.is}}}}
\title{Transportation of diffuse random measures  on $\R^d$}
\date{\today}
\maketitle
\begin{abstract} 
\noindent 
We consider two jointly stationary and ergodic random
measures $\xi$ and $\eta$ on $\R^d$ with equal finite intensities, 
assuming $\xi$ to be diffuse (non-atomic).
An allocation is a  
random mapping taking 
$\R^d$ to $\R^d\cup\{\infty\}$ in a translation invariant way.
We construct allocations transporting  the diffuse
 $\xi$ to  arbitrary 
 $\eta$, 
under the mild condition of existence of an ${ ` }$auxiliary'  point process 
which is needed only in the case when $\eta$ is diffuse. 
When that condition does not hold we show by a counterexample that an allocation 
transporting  $\xi$ to $\eta$ need not exist.\end{abstract}


\vspace{0.2cm}


\section{Introduction}

Mass transportation is an important and lively research area. We refer
to \cite{Villani09} for an extensive monograph on optimal transports.
A more
recent addition to the literature is 
transports between random measures 
(and in particular balancing allocations),  
which connects to the classical topic
in several ways. For instance it was shown in \cite{HoHolPe06,LaTho09}
that balancing transports between stationary random measures
correspond to certain couplings (shift-couplings) of the associated Palm measures,
while \cite{HuesSturm13} studied quantitative optimality
of a balancing allocation. 

Before going further 
we need 
to establish some notation and terminology.
An {\em allocation} 
is a 
random mapping 
$\tau: x\mapsto \tau(x)$ taking
$\R^d$ to $\R^d\cup\{\infty\}$ 
in a translation-invariant (equivariant, covariant) way; 
note that in this paper the invariance is included in 
the definition of allocation.
Let  $\xi$ and $\eta$ be jointly stationary
and ergodic random
measures on $\R^d$ ($d \geq 1$) with finite identical intensities 
$\lambda_\xi = \lambda_\eta$. 
Joint ergodicity  of $\xi$ and $\eta$ means that the distribution of $(\xi,\eta)$ takes only the values 
$0$ and $1$ on translation  invariant sets, and joint
 stationarity 
 means 
that (with  $\overset{D}{=}$  denoting identity in distribution) 
$$ 
 (\xi(x+ \cdot), \eta(x+ \cdot ))    \,\overset{D}{=}\, (\xi,\eta),
 \quad x \in\R^d.
 $$
 Say that an allocation 
{\em balances} the {\em source} $\xi$ and the
{\em destination} 
$\eta$ if 
it transports  $\xi$ to $\eta$, that is,
if (a.s.)\ the image of the measure
$\xi$ under $\tau$ is $\eta$, $$\xi(\tau \in \cdot)=\eta.$$
See~Section~\ref{secprelim} for exact framework and definitions. 
See Remark \ref{history}  for historical notes beyond those in this introduction.

In the present paper we construct allocations
balancing a {\em diffuse} (a {\em non-atomic}) 
source $\xi$
and an
{\em arbitrary} destination $\eta$. In order to explain the 
point
 of the paper, let us 
outline two remarkable
examples.

\vspace{5mm}

\noindent
{\bf Extra head:}
The search for balancing allocations goes back to Liggett's 
surprising idea of `how to choose a head at random' -- an {\em extra head}  --
in a two-sided sequence of i.i.d.\  coin tosses.
If there is a head at the origin, it is an extra head 
(the other coins are i.i.d.). 
If there is a tail at the origin, 
move the origin to the right counting heads and tails until you have more heads than tails.\
Then you are at a head and it is an extra head; see \cite{L02}.

This can be restated as follows. 
If $\eta$ is the simple point process (on $\Z$) 
formed by the heads,
and $T$ is the location of the extra head, then $\eta(T + \cdot)$ has 
the same distribution as the Palm version $\eta^\circ$   of $\eta$,
\begin{align}\label{extra}
\eta(T+\cdot)  \,\overset{D}{=}\, \eta^\circ.
\end{align}
This means that we have constructed a {\em shift-coupling} of
 $\eta$ and $\eta^\circ$.

Allocations provide a proof of this result as follows. 
For each $n \in \Z$ let $\tau(n)$
be the location of the `extra head' found 
when starting from $n$ rather than from the origin $0$. 
Then the map $\tau: \Z \to \Z \cup \infty$ is an allocation
that leaves the heads where they 
are, while $\tau$ turns out to be  (a.s.) a bijection from the tails to the heads.
Thus if we let $\mu$ be the counting measure on $\Z$ 
then $\tau$ balances $\mu$ and $2\eta$ 
(or, equivalently,  $\tau$ balances $\mu$/2 and $\eta$).
According to 
\cite{HoHolPe06,
LaTho09}, 
this implies that \eqref{extra} holds.

\vspace{5mm}

\noindent
{\bf Stable marriage of Poisson and Lebesgue:}
Liggett's idea led Hoffman, Holroyd and Peres to solving an open
problem from the mid nineties: how to find an {\em extra point} in a
Poisson process. Let $\eta$ be a stationary Poisson process on $\R^d$
with intensity $1$.  Associate disjoint cells of volume $1$ to the
points of $\eta$ as follows.  Expand balls simultaneously from all the
points of $\eta$.  If the ball of a particular point has accumulated
volume $1$ before it hits another ball, then this ball is the cell of
that point. If not, then continue expanding to accumulate volume when
reaching space that has not already been reached by another ball; stop
when volume $1$ has been accumulated. It turns out (a.s.) that in this
way each point obtains a cell of volume $1$ and that the cells are
disjoint and cover $\R^d$.  Now let $T$ be the point of the cell
containing the origin.  If the origin is shifted to $T$ and that point
is removed then the remaining points of $\eta(T+\cdot)$ form a
stationary Poisson process with intensity $1$. Thus $T$ is an extra
point; see \cite{HoHolPe06}.

This 
again means that we have constructed a shift-coupling of
 $\eta$ and its Palm version $\eta^\circ$, that is, \eqref{extra} holds.
 And again, allocations  
provide a proof of that result as follows.\ 
For each $x \in \R^d$ let $\tau(x)$ be the point of the cell containing~$x$
and let $\mu$ be Lebesgue measure on $\R^d$.\
Then $\tau$ is (clearly) an allocation balancing $\mu$ and~$\eta$.
According to 
\cite{HoHolPe06,
LaTho09}, 
this implies that \eqref{extra} holds.

\vspace{5mm}

What makes allocations particularly interesting are the shift-coupling 
 results in the above examples.
 In both examples the source $\mu$ is  translation-invariant and non-random,
but the shift-coupling result  
extends to general $\xi$ and $\eta$ as follows.
If a stochastic process 
(or a random measure)   $X$
is stationary and ergodic jointly 
with $\xi$ and $\eta$,  and
$\tau$ transports 
$\xi$ to
 $\eta$, then by shifting the origin to $\tau(0)$ 
the {\em Palm version} of  
$X$ w.r.t.\ $\xi$ 
turns into the 
Palm version of   
$X$
w.r.t.\ $\eta$.
See \eqref{shift-coupling2} in the next section,  
and \cite{LaTho09}, for this result. 
See  \cite{AldousThor,Thor96}
for the origin of shift-coupling.

\vspace{5mm}

\noindent
{\bf Unbiased Skorokhod embedding:}
In \cite{LaMoeTho12, LaTaTho18} it is shown in the one-dimensional case, $d=1$, 
that if $\xi$ is diffuse (non-atomic), $\eta$ is arbitrary,         
and $\xi$ and $\eta$
are mutually singular  
 then they are balanced by
the allocation $\tau$ defined by
\begin{align}\label{alocation}
\tau(x) = \inf\{t>x : \xi([x,t]) \leq \eta([x, t]\},\quad x\in\R.
\end{align}
Local times of Brownian motion are diffuse 
and  (see \cite{GeHo73}) the two-sided standard Brownian motion $(B_x)_{x\in\R}$
 is a Palm version of the $\sigma$-finite stationary Brownian measure. Thus (see  \cite{LaMoeTho12}) the shift-coupling result for general $\xi$ and $\eta$ 
  can be applied 
with $\xi = \ell^0$ the local time at $0$  and 
with 
$\eta = \int\ell^y\nu(dy)$
where $\ell^y$ is the local time at 
$y$ and 
where
$\nu$ is a probability measure without atom at $0$
(to ensure mutual singularity).
This yields the following  {\em unbiased}  Skorokhod embedding: 
 $(B_{\tau(0)+x})_{x\in \R}$ is a two-sided standard Brownian motion
 with distribution $\nu$ at $x=0$. 
 It is said to be {\em unbiased}  
  because 
  not only the one-sided $(B_{\tau(0)+x})_{x\geq 0}$,  
but  also the two-sided $(B_{\tau(0)+x})_{x\in \R}$, is Brownian. 
The same approach results in various 
 embeddings when 
applied to local times associated with Brownian motion, 
e.g.\ extra excursion (see  \cite{LaTaTho18})
and (see \cite{PitmanTang15}) extra Brownian bridge.

\vspace{5mm}

 In the present paper we consider the $d$ dimensional case
when $\xi$ is diffuse, $\eta$ is arbitrary,          
and $\xi$ and $\eta$ need
{\em  not} be mutually singular. 
It turns out that there are  special cases where balancing allocations 
do {\em not} exist (Section~\ref{seccounter}).
In order to guarantee the existence of a balancing allocation 
we impose the mild condition of the existence of a 
non-zero simple point process $\chi$
on $\R^d$ with  finite intensity $\lambda_\chi$ and
such that  $\xi$, $\eta$ and $\chi$
are jointly stationary and ergodic. 
We call this simple point process $\chi$ {\em auxiliary}. 

The following theorem is the main result of the paper. 
Note that the auxiliary $\chi$ is only needed when $\eta$ is purely diffuse.

\begin{theorem}\label{main} 
Assume that $\xi$ and $\eta$
are 
jointly stationary and ergodic random measures on $\R^d$.
Let $\xi$ be diffuse (non-atomic) and 
$$
0<\lambda_\xi=\lambda_\eta<\infty.
$$ 
Then there exists an allocation balancing $\xi$ and $\eta$ if one of the following conditions holds:

{\em (a)} $\eta$ has a non-zero discrete component;

{\em (b)} $\eta$ is diffuse and there exists an auxiliary  $\chi$.
\end{theorem}

Condition (a) covers discrete $\eta$ and, in particular, point processes. 
Note that under condition (a) an auxiliary $\chi$ always exists and can  be chosen 
as a {\em factor} of $\eta$, that is, as a measurable and equivariant
(w.r.t.\ translation) function of $\eta$. 
When the discrete component of $\eta$
has isolated atoms then $\chi$ can be taken to be 
the support of $\eta$. 
And although in general the support of the discrete component 
need not consist of isolated points (it can even be dense), 
there exists a constant $c > 0$  such that the following simple point process 
(here $\delta_x$ is the 
measure with mass $1$ at $x$)
\begin{align}\label{i3}
\chi = \sum_x  \I\{\eta(\{x\})> c\} \delta_x
\end{align}
is non-zero. 

Under condition (b), there are also cases where an auxiliary $\chi$ exists 
as a factor of $(\xi,\eta)$. But the counterexample in Section 8 
shows that the condition of the existence of an auxiliary $\chi$ cannot simply be removed from (b). 
There are diffuse $\xi$ and $\eta$ such that an allocation transporting $\xi$ to $\eta$ does not exist. 
However, if extension of the underlying probability space
$(\Omega, \mathcal{F}, \BP)$ is allowed, then that obstacle can be overcome.

\begin{corollary}\label{cmain}
Assume that $\xi$ and $\eta$
are
jointly stationary and ergodic random measures on $\R^d$. 
Let $\xi$ be diffuse and $$
0<\lambda_\xi=\lambda_\eta<\infty.
$$ 
Extend $(\Omega, \mathcal{F}, \BP)$ to support a stationary Poisson process
 $\chi$ on $\R^d$
which is independent of $\xi$ and~$\eta$.
Then there exists an allocation balancing $\xi$ and $\eta$.
\end{corollary}

As stressed above, an important property of a balancing allocation $\tau$ is the
shift-coupling at \eqref{shift-coupling2}.
 When specialised  to 
the case  
$\xi = \lambda_\eta$ times  
Lebesgue measure, then  \eqref{shift-coupling2}
turns into \eqref{shift-coupling2222},
that is,  
into \eqref{extra} 
with $T=\tau(0)$. 
This is called {\em extra head scheme} in \cite{HP05}. 
Note however that removing a `head' (or some pattern like an excursion)
 from $\eta^{\circ}$, does not
 result  in a copy of $\eta$ except in very special cases such as 
  coin tossing, the Poisson process and Brownian motion.

According to Theorem~\ref{main} applied with 
$\xi = \lambda_\eta$ times 
Lebesgue measure,
the result \eqref{extra} always holds with $T$ a function of $\eta$
alone, unless $\eta$ is diffuse. In that case, 
 \eqref{extra} still holds with $T$ a function of $\eta$,
provided there exists an auxiliary $\chi$ that is a factor of $\eta$. 
Finally, according to Corollary~\ref{cmain}, if external randomisation is allowed then \eqref{extra} 
always holds for some~$T$, also when $\eta$ is diffuse.
(Actually, due to ergodicity and the definition of Palm probabilities at  \eqref{Palmmeasure}, 
the distributions of $\eta$ and $\eta^{\circ}$ have the same (zero-one)  values on invariant sets. 
Thus, according to an abstract 
existence result in \cite{Thor96} for shift-coupling on groups,
 \eqref{extra} always holds for some $T$ defined on an extended probability space.
But the  constructions of $T$ in the present paper are explicit.)

The plan of the paper is as follows. Section~\ref{secprelim} collects
some preliminaries on stationary random measures, 
balancing allocations, and Palm theory.
Section~\ref{secdiscretedest}  prepares for the proof of 
Theorem~\ref{main}. The theorem  is then proved in  
Sections~\ref{secatomsdest}-\ref{secmixeddest}, 
where 
allocations are constructed in four exhaustive cases: 
discrete $\eta$ with only isolated atoms in Section~\ref{secatomsdest}, 
discrete $\eta$ with some accumulating atoms in Section~\ref{secdiscrete}, 
diffuse $\eta$  ($\eta$ with no atoms)  in Section~\ref{secdifdest}, 
and $\eta$ 
 with discrete and  diffuse parts in Section~\ref{secmixeddest}.
We give algorithmically 
explicit constructions in the discrete 
cases, and less explicit in the diffuse.
Section~\ref{seccounter} proves by counterexample that 
the auxiliary $\chi$ cannot simply be removed from part (b) of Theorem~\ref{main}. 
Section~\ref{secrem} concludes with remarks.

\section{Preliminaries}\label{secprelim}

Let $(\Omega,\mathcal{F},\BP)$ be a probability space
with expectation operator $\BE$.
A {\em random measure} (resp.\ {\em point process}) $\xi$ on $\R^d$ (equipped
with its Borel $\sigma$-field $\mathcal{B}(\R^d)$) is a kernel 
from $\Omega$ to $\R$ such that $\xi(\omega,C)<\infty$
(resp.\ $\xi(\omega,C)\in\N_0$)
for $\BP$-a.e.\ $\omega$ and all compact $C\subset\R^d$; see
e.g.\ \cite{Kallenberg,LastPenrose17}.
A point process $\xi$ is called {\em simple} if
$\xi(\{x\})\in\{0,1\}$, $x\in\R^d$, except on a set
with probability zero.
Further, let
$(\Omega,\mathcal{F})$ be 
equipped with a {\em measurable flow}
$\theta_x\colon\Omega \to \Omega$, $x\in \R^d$. This is a family
of  mappings such that $(\omega,x)\mapsto \theta_x\omega$ 
is measurable, $\theta_0$ is the identity on $\Omega$ and
\begin{align}\label{flow}
  \theta_x \circ \theta_y =\theta_{x+y},\quad x,y\in \R^d,
\end{align}
where $\circ$ denotes composition.
An {\em allocation} \cite{HP05,LaTho09} is a measurable mapping
$\tau\colon\Omega\times\R^d\rightarrow\R^d\cup\{\infty\}$ that is {\em equivariant}
in the sense that
\begin{align}\label{allocation}
\tau(\theta_y\omega,x-y)=\tau(\omega,x)-y,
\quad  x,y\in\R^d,\, \text{$\BP$-a.e.\ $\omega\in\Omega$}.
\end{align}
We illustrate these concepts with a simple but illustrative example.

\begin{example}\rm
Take $\Omega$
as the space of all locally finite  sets $\omega\subset\R^d$,
equipped with the usual $\sigma$-field and define 
$\theta_x\omega:=\omega-x$, $x\in\R^d$.
The counting measure $\xi(\omega)$ supported by $\omega$ defines
a (discrete) random measure $\xi$. An example of an allocation $\tau$
is to take $\tau(\omega,x)$ as the point of $\omega\in\Omega$
closest to $x\in\R^d$, using lexicographic order to break ties.
(For $\omega=\emptyset$ we set $\tau(\omega,\cdot)\equiv\infty$.)
\end{example}

We assume that the measure $\BP$ is {\em stationary}; that is
$$
\BP\circ\theta_x=\BP,\quad x\in\R^d,
$$
where $\theta_x$ is interpreted as a mapping from $\mathcal{F}$ to $\mathcal{F}$
in the usual way:
$$
\theta_xA:=\{\theta_x\omega:\omega\in A\},\quad A\in\mathcal{F},\, x\in\R^d.
$$
A random measure $\xi$ on $\R^d$ is said to be 
{\em stationary}  if
\begin{align}\label{adapt}
\xi(\theta_x\omega,C-x)=\xi(\omega,C),\quad C\in\mathcal{B}(\R^d),\,x\in\R^d,
\BP\text{-a.e.\ $\omega\in\Omega$}.
\end{align}
Abusing our notation by defining the shifts $\theta_x$, $x\in\R^d$,
also for measures on $\R^d$
in the obvious way, we obtain from \eqref{adapt} and stationarity of $\BP$ that
$$
\theta_x\xi \overset{D}{=}\xi, \quad x\in \R^d.
$$
The {\em invariant} $\sigma$-field $\mathcal{I}\subset \mathcal{F}$ is the
class of all sets $A\in\mathcal{F}$ satisfying $\theta_xA=A$ for all $x\in\R^d$.
We also assume that $\BP$ is {\em ergodic}; that is for any
$A\in\mathcal{I}$, we have $\BP(A)\in\{0,1\}$ (see however Remark~\ref{ergod}).

Let  $\xi$ be a stationary random measure on $\R^d$ 
with positive and finite  {\em intensity} $$\lambda_\xi:=\BE\xi[0,1]^d.$$
The {\em Palm probability measure} $\BP_\xi$ of $\xi$
(with respect to $\BP$) is defined by
\begin{align}\label{Palmmeasure} 
\BP_\xi(A):=\lambda_\xi^{-1}\lambda_d(B)^{-1}\,\BE \int \I_B(x)\I_A(\theta_x)\, \xi(dx),
\quad A\in\mathcal{F},
\end{align}
where $B\subset\R^d$ is a Borel set with positive and finite Lebesgue 
measure $\lambda_d(B)$. This definition does not depend on $B$.
The expectation operator associated with $\BP_\xi$ is denoted by $\BE_\xi$.
Any multiple $c\lambda_d$ of  Lebesgue measure
 is a (rather trivial) stationary random
measure. In this case we obtain from stationarity of $\BP$ that 
\begin{align}\label{LebPalm}
\BP_{c\lambda_d} = \BP.
\end{align}

An allocation $\tau$ {\em balances} two random measures $\xi$ and 
$\eta$ if
\begin{align}\label{bal456}
\BP(\xi(\{s\in\R:\tau(s)=\infty\})>0)=0
\end{align}
and the image measure of $\xi$ under $\tau$  is $\eta$, that is, 
\begin{align}\label{bal123}
\int \I\{\tau(s)\in C\}\,\xi(ds)=\eta(C),\quad C\in\mathcal{B}(\mathbb{R}^d),\, \BP\text{-a.e.}
\end{align} 
The balancing properties \eqref{bal456} and \eqref{bal123} imply easily that
\begin{align}\label{i2}
\lambda_\xi= \lambda_\eta.
\end{align}
By \cite[Theorem 4.1]{LaTho09} we then have the {\em shift-coupling}
\begin{align}\label{shiftc}
\BP_\xi(\theta_{\tau(0)}\in\cdot)=\BP_\eta.
\end{align}

\begin{remark}\label{Palmprel}\rm
In this remark we consider the shift-coupling result
 \eqref{shiftc}
 in terms of random elements. 
 Let $X$ be a random element that can be translated  
by $t\in \R^d$, for instance a random measure, or a random field, 
or the identity on $\Omega$.
Then $X$, $\xi$ and $\eta$,
defined on the probability space $(\Omega,\mathcal{F},\BP)$  above,
 are jointly stationary and ergodic.

A {\em Palm version of $X$ w.r.t.\ $\xi$} is any random element 
with the distribution $\BP_\xi(X\in\cdot)$. 
In particular, according to \eqref{LebPalm}, 
 a Palm version of $X$ w.r.t.\  Lebesgue measure is $X$ itself. 
 What is generally called a  {\em Palm version} of a random measure 
 $\xi$
  is a random measure $ \xi^{\circ} $ with the distribution $\BP_\xi(\xi\in\cdot)$. 
  That is, $ \xi^{\circ} $ is a Palm version of $\xi$ with respect to itself. 
 
 The informal interpretation of $ \xi^{\circ} $ in the case when $ \xi $ is a simple 
  point process is that $ \xi^{\circ} $ behaves like $ \xi$ 
  conditioned on having a point at the origin.
 When 
  $ \xi$ is a Poisson process 
  then  (and only then) $ \xi^{\circ} $ can be obtained simply
  by placing an extra point at the origin, $ \xi^{\circ} = \xi + \delta_0$.
   In the ergodic case, which is assumed here, another 
  informal interpretation
  of  $ \xi^{\circ} $  is 
  that  $ \xi^{\circ} $ behaves  like $\xi$ with origin shifted to 
  a uniformly chosen point of $\xi$, -- or when $\xi$ is
  not a simple point process, 
  to a uniformly chosen location in the 
  mass of $\xi$. 
  (These interpretations are motivated by limit theorems.)

 From \eqref{shiftc} we obtain  
 the following shift-coupling result.  
 If $\tau$ is an allocation balancing  
$\xi$ and $\eta$, then a shift of the origin to $\tau(0)$
turns a Palm version of $X$ w.r.t.\ $\xi$  into 
a Palm version of $X$ w.r.t.\ $\eta$,
\begin{align}\label{shift-coupling2}
\BP_\xi(\theta_{\tau(0)}X\in\cdot)=\BP_\eta(X\in \cdot). \end{align}
In particular, from  
 \eqref{LebPalm} and \eqref{shift-coupling2} 
  with $X = \eta$, we obtain the following 
   result.  
If $\tau$ balances  the Lebesgue-measure multiple $\lambda_\eta\lambda_d$  and $\eta$, 
then a shift of the origin to $\tau(0)$
turns the stationary $\eta$
into a Palm version $\eta^\circ$ of $\eta$,
\begin{align}\label{shift-coupling2222}
\theta_{\tau(0)} \eta
 \,\overset{D}{=}\, \eta^{\circ}.
 \end{align}
On the other hand, 
 from  
 \eqref{LebPalm} and \eqref{shift-coupling2}
  with $X = \xi$, we obtain  the reverse result.
 If $\tau$ balances  
$\xi$ and  the Lebesgue-measure multiple $\lambda_\xi\lambda_d$,
 then a shift of the origin to $\tau(0)$
  turns the  a Palm version $\xi^\circ$ of $\xi$ 
into the stationary $\xi$, 
\begin{align}\label{shift-coupling2224}
\theta_{\tau(0)} \xi^\circ
 \,\overset{D}{=}\, \xi.
\end{align}
If 
$\xi$ is the only source of randomness,
then, as a rule, allocations  balancing 
$\xi$ and a multiple of Lebesgue measure do not exist, see \cite{HP05} and also 
Section \ref{seccounter}. 
\end{remark}

\section{Allocations to Discrete Random Measures}\label{secdiscretedest}

Let $\xi$ and $\eta$ be two jointly stationary and ergodic random measures
on $\R^d$; see Section~\ref{secprelim}. For the remainder of
  this paper we assume $\lambda_\xi>0$ and $\lambda_\eta>0$.  
Assume that $\eta$ is a discrete random measure with locally finite support.
Let $\eta^*$ be the simple point process with the same support as $\eta$.
Assume also that $\xi$, $\eta$ and $\eta^*$ have positive and finite
intensities $\lambda_\xi$, $\lambda_\eta$ and $\lambda_{\eta^*}$ respectively.
(The assumptions  $\lambda_{\eta^*}<\infty$ has been made by convenience and could be removed.)
We consider an allocation $\tau$ with the property
\begin{align}\label{e3.59}
\xi(\{z\in\R^d:\tau(z)\notin \eta^*\cup\{\infty\}\})=0,\quad \BP\text{-a.s.}
\end{align}
Define
\begin{align*}
C^\tau(z):=\{y\in\R^d\colon\tau(y)=z\}, \quad z\in\R^d.
\end{align*}
Note that $C^\tau(z)$ is random.

Let $\alpha\in(0,\infty)$. 
The allocation $\tau$ is said to have {\em appetite} $\alpha$ 
(w.r.t.\ $(\xi,\eta)$) if \eqref{e3.59} and the following two properties hold.
First we have almost surely that 
\begin{align}\label{e11.45}
\eta^*(\{x\in\R^d:\xi(C^\tau(x))>\alpha\eta\{x\}\})=0.
\end{align}
Second the probability that
\begin{align}\label{e11.46}
\xi(\{z\in\R^d:\tau(z)= \infty\})>0\quad\text{and}\quad
\eta^*(\{x\in\R^d:\xi(C^\tau(x))<\alpha\eta\{x\}\})>0
\end{align}
is zero.

\begin{proposition}\label{p1} 
Assume that the allocation $\tau$
has appetite $\alpha$ for some $\alpha\in (0,\lambda_\xi\lambda_\eta^{-1}]$.
Then $\tau$ is $\alpha$-balanced, that is we have a.s.\ that
$\eta(\{x\in\R^d:\xi(C^{\tau}(x))\ne \alpha\eta\{x\}\})=0$.
Moreover, we have that $\lambda_\xi\BP^0_\xi(\tau(0)\ne\infty)=\alpha\lambda_\eta$.
\end{proposition}
{\em Proof:} We generalise the proof of \cite[Theorem 10.9]{LastPenrose17}.
We start with a general result, that might be of independent interest.
Let $g\colon \Omega\times\Omega\to[0,\infty)$. Then
\begin{align}\label{e3.1}
\lambda_\xi\BE_\xi\I\{\tau(0)\ne\infty\}g(\theta_0,\theta_{\tau(0)})
=\lambda_{\eta^*}\BE_{\eta^*}\int_{C(0)} g(\theta_x,\theta_0)\,\xi(dx),
\end{align}
where we abbreviate $C(z):=C^\tau(z)$, $z\in\R^d$.
This follows from Neveu's exchange formula (see e.g.\ 
\cite[Remark 3.7]{LastPenrose17}) applied to the function
$h(\omega,x):=g(\omega,\theta_x\omega)\I\{\tau(\omega,0)=x\}$
(and replacing $(\eta,\xi)$ by $(\xi,\eta^*)$). 

Let 
\begin{align*}
A:=\{\text{there exists $x\in\eta$ such that $\xi(C(x))<\alpha\eta\{x\}$}\}.
\end{align*}
This event is invariant.
It follows from \eqref{e11.46} that
\begin{align*}
\BP_\xi(A)=\BP_\xi(\{\tau(0)\ne\infty\}\cap A).
\end{align*}
Therefore we obtain from \eqref{e3.1} that
\begin{align*}
\lambda_\xi\BP_\xi(A)=\lambda_{\eta^*}\,\BE_{\eta^*}\I_A\xi(C(0)).
\end{align*}
By definition \eqref{Palmmeasure} of the Palm 
probability measure of $\eta^*$ we hence obtain
for each Borel set $B\subset\R^d$ with $0<\lambda_d(B)<\infty$
that
\begin{align}\label{e3.4}\notag
\lambda_\xi\BP_\xi(A)
&=\lambda_d(B)^{-1} 
\BE \int_B\I\{\theta_x\in A\}
\xi\circ\theta_x(C(0,\theta_x))\,\eta^*(dx)\\
&=\lambda_d(B)^{-1} \BE\I_A\int_B\xi(C(x))\,\eta^*(dx),
\end{align}
where we have used the  invariance of $A$ and 
\begin{align*}
\xi\circ\theta_x(C(\theta_x,0))&=\int\I\{\tau(\theta_x,y)=0\}\,\xi(\theta_x\omega,dy)
=\int\I\{\tau(y+x)=x\}\,\xi(\theta_x\omega,dy)\\
&=\int\I\{\tau(y)=x\}\,\xi(dy)=\xi(C(x)).
\end{align*}
Using \eqref{e11.45} and denoting the invariant $\sigma$-field by $\mathcal{I}$, this yields
\begin{align}\label{e3.5}
\lambda_\xi\BP_\xi(A)&\le \lambda_d(B)^{-1}\alpha\,\BE[\I_A\eta(B)]\\
\notag
&= \lambda_d(B)^{-1}\alpha\,\BE[\I_A\BE[\eta(B)\mid \mathcal{I}]]\\
\notag
&=\alpha\lambda_\eta\,\BP(A)\le \lambda_\xi\BP(A)=\lambda_\xi\BP_\xi(A), 
\end{align}
where we have used ergodicity to get the second equality (almost surely)
and the assumption $\alpha\le\lambda_\xi\lambda^{-1}_\eta$ to get the second inequality.
Therefore the above inequalities are in fact equalities. Hence
\eqref{e3.4} and the right-hand side of \eqref{e3.5} coincide, yielding that
\begin{align*}
\BE\I_A\int_B (\alpha\eta\{x\}-\xi(C(x))\,\eta^*(dx)=0.
\end{align*}
Taking $B\uparrow\R^d$ and using montone convergence (justified by \eqref{e11.45}), this yields
\begin{align*}
  \I_A\int(\alpha\eta\{x\}-\xi(C(x)))\,\eta^*(dx)=0,
\quad \BP\text{-a.s.}
\end{align*}
Hence we have $\BP$-a.s.\ on $A$ that $\xi(C(x))=\alpha\eta\{x\}$ for all $x \in \eta^*$.
By definition of $A$ this
is possible only if $\BP(A)=0$. This implies the first assertion.

To prove the second assertion we use \eqref{e3.1} with $g\equiv 1$ to obtain that
$\lambda_\xi\BP_\xi(\tau(0)\ne\infty)=\lambda_{\eta^*}\alpha\, \BE_{\eta^*}\eta\{0\}$.
Since it follows straight from the definitions that
$\lambda_{\eta^*}
\BE_{\eta^*}\eta\{0\}=\lambda_\eta$ we can conclude the assertion.
\qed

\section{Destination  Isolated Atoms}\label{secatomsdest} 

The spatial version of the Gale–Shapley allocation  
introduced in 
\cite{HoHolPe06} balances Lebesque measure to a simple point process. 
The simplified description of it in the introductory
Poisson-Lebesgue example 
 is not an effective way of proving that it actually works, 
 the efficient way is algorithmic.
We now extend this allocation to  balance a diffuse $\xi$ to an
$\eta$ consisting of  isolated atoms. The extension is needed because 
unlike the Lebesgue measure a diffuse measure can have positive 
mass on lower dimensional sets like the boundaries  of balls.
Motivated by the point-optimal stable allocation introduced
in \cite{HoHolPe06}, we formulate an algorithm providing
an allocation of appetite $\alpha$.

The idea behind the algorithm (in the case $\alpha =1$) can be sketched as follows. 
In the first round of the algorithm assign a {\em preference set} 
to each $\eta$-atom, that is, a set of sites in $\R^d$ that the atom {\em proposes} to. 
We do this by a finite recursion 
(note that the following  four items 
can be reduced to one item when $\xi$ is Lebesgue measure):

\begin{itemize}

\item    From each $\eta$-atom blow up a "first" closed ball until you have gathered  
    $\xi$-mass 
     at least equal to the mass of that atom, $m_1$ say.

\item Put $m_2 = m_1 $ minus the $\xi$-mass in the {\bf interior} of the "first" ball.  
This remaining mass $m_2$ is thus part of a $\xi$-mass sitting on the boundary of a $d$ dimensional ball.

\item

    Then blow up a "second" closed ($d$ dimensional) ball from (e.g.)\ the lexicographically lowest 
    location on the boundary of the "first" ball 
    (think of it as a pole) until you have gathered $\xi$-mass on the
    {\bf boundary} 
    of the "first" ball that is at least $m_2$. This mass $m_2$ is sitting on a closed cap of the 
   boundary of the "first" ball. Put $m_3 = m_2$ minus the $\xi$-mass of the
   interior of that cap (the relative interior of the cap w.r.t. the boundary, 
  the sphere).
   This remaining mass, $m_3$,  is a part of a mass sitting on 
    the boundary of a $d-1$ dimensional ball.

\item   

    Repeat this down the dimensions until the $\eta$-atom has a {\em preference set} 
  (the union of these interior sets) 
  of $\xi$-mass exactly $m_1$ (because the final cap will be  
   a circle segment and its boundary will have at most two points and their $\xi$-mass is zero 
   since $\xi$ is diffuse).
    Note that the preference sets of different $\eta$-atoms may overlap. 
    Note also that a preference set of an $\eta$-atom can contain other 
$\eta$-atoms. 
\end{itemize}
Now let each site that lies in at least one preference set 
of an $\eta$-atom put the closest of those atoms on a shortlist, using
lexicographic order to break ties.
The atoms associated with the other preference sets are rejected.
Each atom has now a rejection set, a subset of its preference 
set containing sites that rejected the proposal.

Repeat   the above procedure recursively 
by blowing up a ball around each $\eta$-atom restricted
to the complement of its associated rejection set
(thus extending its preference set)
and, after each round,  add the new rejections to its rejection set.
One of two things can happen for a site $z\in\R^d$. Either it
never appears in one of the preference sets. Then
$z$ has no partner and is allocated to $\infty$.
Or it eventually shortlists a
single point $x$. Then $z$ is allocated to $x$.

\bigskip

In the algorithm we will use the notation 
$D(t-) :=\bigcup_{s<t}D(s)$
for an increasing family of sets $D(s)\subset \R^d$, $s>0$.
If $\mu$ is a measure on $\R^d$ and $\mu\{x\}>0$
we write $x\in\mu$.

\begin{algorithm}\label{a11.5}\rm
Let $\alpha>0$, $\mu\ne 0$ be discrete with isolated atoms 
and $\nu$ be diffuse with infinite mass.
For $n\in\N_0$, $x\in\mu$ and $z\in\R^d$, 
define the sets 
\begin{align*}
C_n(x) &\subset \R^d\qquad  \text{(the
set of sites {\em claimed}, or {\em preferred}, by $x$ at stage $n$),}\\
R_n(x) &\subset \R^d  \qquad \text{(the set of sites
{\em rejecting} $x$ during the first $n$ stages),}\\
A_n(z) &\subset \mu \qquad\;\; \text{(the set of points of $\mu$
claiming site $z$ in the first $n$ stages),}
\end{align*}
via the following recursion in $n$.
Start with setting $R_0(x):=\emptyset$. 

\begin{enumerate}
\item This first step is a recursion within the recursion. 
Fix $x\in\mu$ and mostly suppress it in the notation until 
step (1) is over. 
Define 
\begin{align*}
S_{n,1}(s) &:= B(x,s) = 
 \text{
the
ball with center $x$ and radius $s>0$,} \\
\beta_{n,1} &:= \alpha\mu\{x\}, \\
s_{n,1} 
&:= \inf\{s\geq 0 : \nu\big(S_{n,1}(s) \setminus R_{n-1}\big) \geq \beta_{n,1}\}, \\
\Delta_{n,1} &:= S_{n,1}(s_{n,1})\setminus S_{n,1}(s_{n,1}-)
\qquad \text{note that } \Delta_{n,1} := \partial B(x, s_{n,1}).
\end{align*}
For $k=1,\dots,d-1$, proceed recursively as follows. Let 
$y_{n,k} \in \R^d$ be the lexicographically lowest element of 
$
\Delta_{n,k}
$
and set
\begin{align*}
S_{n{,k+1}}(s) &:= B(y_{n,k},s) \cap \Delta_{n,k} \\
\beta_{n,k+1} &:= \beta_{n,k} - \nu\big(S_{n,k}(s_{n,k}-)\setminus R_{n-1}\big)\\
s_{n,k+1} 
&:= \inf\{s\geq 0 : \nu\big(S_{n,k+1}(s) \setminus R_{n-1}\big) \geq \beta_{n,k+1}\}\\
\Delta_{n,k+1} &:= S_{n,k+1}(s_{n,k+1})\setminus S_{n,k+1}(s_{n,k+1}-).
\end{align*}
Then  $\Delta_{n,d}$
contains at most two elements. 
Since $\nu$  is diffuse this implies that 
$$
\nu\big(S_{n,d}(s_{n,d}) \setminus R_{n-1}\big)=
\nu\big(S_{n,d}(s_{n,d}-) \setminus R_{n-1}\big)=\beta_{n,d}.
$$
Now set
$$
C_{n} := S_{n,1}(s_{n,1}-)\cup \dots \cup S_{n,d}(s_{n,d}-).
$$
Since $C_n$ is a disjoint union we have
$$
\nu(C_{n}\setminus R_{n-1}) = \nu(S_{n,1}(s_{n,1}-)\setminus R_{n-1})+ \dots 
+\nu(S_{n,d}(s_{n,d}-)\setminus R_{n-1}).
$$
Since
  $$
\beta_{n,k+1} := \beta_{n,k} - \nu\big(S_{n,k}(s_{n,k}-)\setminus R_{n-1}\big)
 \qquad 
\text{and} \qquad
\nu\big(S_{n,d}(s_{n,d}-) \setminus R_{n-1})\big)=\beta_{n,d}
$$
this yields
$$
\nu(C_{n}\setminus R_{n-1}) = (\beta_{n,1}-\beta_{n,2})+ \dots 
+ (\beta_{n,d-1}-\beta_{n,d})
+\beta_{n,d} = \beta_{n,1}.
$$
Thus, due to $
\beta_{n,1} = \alpha\mu\{x\}$, we obtain
$$
\nu(C_{n}\setminus R_{n-1}) =\alpha\mu\{x\}.
$$
\item Recall that $x$ was suppressed in the above step. We
 now make it explicit and write $C_n(x)$ instead of only $C_n$.  For $z\in\R^d$, define
\begin{align*}
A_{n}(z):=\{x\in\mu: z\in C_{n}(x)\}.
\end{align*}
If $A_{n}(z)\ne\emptyset$ then define 
$$
\tau_{n}(z):=l(\{x\in A_n(z):\|z-x\|=d(z,A_{n}(z))\})
$$
as the point {\em shortlisted} by site $z$ at stage $n$,
where $l(B)$ denotes the lexicographic minimum of a finite non-empty 
set $B\subset\R^d$ and where $d(z,A_{n}(z))$ is the distance of $z$ 
from the set $A_{n}(z)$. If $A_{n}(z)=\emptyset$ then define $\tau_{n}(z):=\infty$.
\item For $x\in\mu$, define
$$
R_{n}(x):=\{z\in C_{n}(x):\tau_{n}(z)\ne x\}.
$$

\end{enumerate}

\end{algorithm}

\bigskip

Now define a mapping $\tau^{\alpha}(\nu,\mu,\cdot)\colon\R^d\to\R^d\cup\{\infty\}$ as follows.
If $\tau_n(z)=\infty$ 
for all $n\in\N_0$ put
$\tau^{\alpha}(\nu,\mu,z):=\infty$. 
Otherwise, set $\tau^{\alpha}(\nu,\mu,z):=\lim_{n\to\infty}\tau_n(z)$.
We argue as follows that this limit exists. 
Defining $C_0(x):=\{x\}$ for all $x\in\mu$, we assert that for
all $n \in \N$ the following holds:
\begin{align}\label{1}
C_n(x)&\supset C_{n-1}(x),\quad   x\in\mu,\\  
\label{2}
A_n(z) &\supset A_{n-1}(z),\quad z \in \R^d,\\
\label{3}
R_n(x) &\supset R_{n-1}(x),\quad x\in\mu.
\end{align}
This is proved by induction; clearly \eqref{1} implies
\eqref{2} and \eqref{2} implies \eqref{3}, while \eqref{3} implies that 
\eqref{1} holds for the next value of $n$.
By \eqref{2},  $\|\tau_n(z)-z\|$ is decreasing in $n$,  and hence,
since $\mu$ is locally finite,
there exist $x\in\mu$ and $n_0\in\mathbb{N}$ such that
$\tau_n(z)=x$ for all $n\ge n_0$. In this case we define
$\tau^{\alpha}(\nu,\mu,z):=x$. 
If $\nu(\R^d)<\infty$ or
$\mu(\R^d)=0$ we set $\tau^{\alpha}(\nu,\mu,z):=\infty$.
We shall now prove that $\tau^{\alpha}$ (applied with $\xi$ and 
$\eta$ instead of $\nu$ and $\mu$)
has the following property defined in Section 3.

\begin{lemma}\label{l11.6} 
Assume that $\xi$ and $\eta$
are jointly stationary and ergodic random measures on $\R^d$ such that
$\xi$ is diffuse, $\eta$ is discrete with locally finite support 
and $\lambda_\xi\lambda_\eta>0$. Let $\alpha>0$. Then $\tau$ defined on $\Omega\times\R^d$
by $\tau(\omega,x):=\tau^\alpha(\xi(\omega),\eta(\omega),x)$
is an allocation with appetite $\alpha$.
\end{lemma}
\begin{proof} It follows by induction over the
stages of Algorithm \ref{a11.5} that the mappings $\tau_n$ are
measurable as functions of $\nu$, $\mu$ and $z$,
where measurability in $\nu$ and $\mu$ refers to the standard
$\sigma$-field on the space of locally finite measures; see e.g.\ \cite{LastPenrose17}. 
(The proof of this fact is left to the reader.) 
Hence $\tau^{\alpha}$ is measurable. Moreover it is clear that
$\tau^{\alpha}$ and hence also $\tau$ has the required covariance property.
Next we note that $\BP(\xi(\R^d)=\eta(\R^d)=\infty)=1$,
a consequence of ergodicity and $\lambda_\xi\lambda_\eta>0$.

In the remainder of the proof  we fix two locally finite 
measures $\nu$ and $\mu\ne 0$. We assume that $\nu$ is diffuse
and satisfies $\nu(\R^d)=\infty$, while $\mu$ is  assumed
to be discrete with purely isolated atoms.
Upon defining $\tau^{\alpha}(\nu,\mu,\cdot)$  we noted that for each $z\in\R^d$, 
either $\tau^{\alpha}(\nu,\mu,z)=\infty$ or $\tau_n(z)=x$
for some $x\in\mu$ and all sufficiently large $n\in\N$. Therefore
\begin{align}\label{e11.876}
\I\{\tau^{\alpha}(\nu,\mu,z)=x\}
=\lim_{n\to\infty}\I\{z\in C_n(x)\setminus R_{n-1}(x)\},\quad z\in\R^d.
\end{align}
On the other hand, by Algorithm \ref{a11.5}(1) we have
$\nu(C_n(x)\setminus R_{n-1}(x))\le\alpha\mu\{x\}$, so that \eqref{e11.45} follows
from Fatou's lemma.

As in Section 3 we set
\begin{align*}
C^{\tau^\alpha}(x):=\{z\in\R^d:\tau^\alpha(\nu,\mu,z)=x\}.
\end{align*}
We now show that $\{z\in\R^d:\tau^\alpha(\nu,\mu,z)=\infty\}\ne\emptyset$ and
  $\{x\in\mu:\nu(C^{\tau^\alpha}(x))<\alpha\mu\{x\}\}\ne\emptyset$
cannot hold simultaneously, 
implying  the event at \eqref{e11.46} to have probability zero.
For that purpose we assume the strict inequality
$\nu(C^{\tau^\alpha}(x))<\alpha\mu\{x\}$ for some $x\in\mu$. 
By \eqref{e11.876} this implies that there exist $n_0\in\N$ and $\alpha_1<\alpha$ such that
$\nu(C_n(x)\setminus R_{n-1}(x))\le \alpha_1\mu\{x\}$ for 
$n\ge n_0$.
Let $C_\infty(x):=\cup^\infty_{n=1}C_n(x)$. We assert that
$C_\infty(x)=\R^d$. Assume on the contrary this is not the case.
By construction, there exist $r_n(x)>0$, $n\in\N$, such that $B^0(x,r_n(x))\subset C_n(x)\subset B(x,r_n(x))$,
where $B^0(x,r_n(x))$ is the interior of $B(x,r_n(x))$.
Since $C_\infty(x)\ne\R^d$ and the sets $C_n$ are increasing,
we have $r_\infty(x):=\lim_{n\to\infty}r_n(x)<\infty$.
Then $C_\infty(x)\subset B(x,r_\infty(x))$ is bounded and there exists $n\ge n_0$ such
that $\nu(C_\infty(x)\setminus C_n(x))\le \mu\{x\}(\alpha-\alpha_1)/2$.
Hence we obtain (since $R_{n-1}(x)\subset C_n(x)$)
$$
\nu(C_\infty(x)\setminus R_{n-1}(x))
= \nu(C_\infty(x)\setminus C_n(x))+\nu(C_n(x)\setminus R_{n-1}(x))\le \alpha_2\mu\{x\},
$$
where $\alpha_2:=(\alpha+\alpha_1)/2<\alpha$.
By definition of the algorithm this implies that  $C_\infty(x)$ is a strict subset of $C_n(x)$. This contradiction shows  
that $C_\infty(x)=\R^d$.
Now taking $z\in\R^d$, we hence have $z\in C_n(x)$ 
for some $n\ge 1$, so that $z$ shortlists either $x$ or some
closer point of $\mu$. In either case, 
$\tau^\alpha(\nu,\mu,z)\ne\infty$.
\end{proof}

\begin{proposition}\label{plfdisdest}
Assume that $\xi$ and $\eta$
are jointly stationary and ergodic random measures on $\R^d$ such that
$\xi$ is diffuse, $\eta$ is discrete with locally finite support (isolated atoms)
and $0<\lambda_\xi=\lambda_\eta<\infty$. 
Then $\tau$ defined on $\Omega\times\R^d$
by $\tau(\omega,x):=\tau^1(\xi(\omega),\eta(\omega),x)$
is an allocation balancing $\xi$ and $\eta$.
\end{proposition}
{\em Proof:} By Lemma \ref{l11.6} 
$\tau$ is an allocation of appetite 1. Since $\lambda_\xi=\lambda_\eta$
we can apply Proposition \ref{p1} to see that
$\eta(\{x\in\R^d:\xi(C^{\tau}(x))\ne \eta\{x\}\})=0$ holds almost surely
and moreover that $\BP_\xi(\tau(0)\ne\infty)=1$.
These two facts imply the desired balancing property of $\tau$.\qed

\bigskip

The allocation in Proposition \ref{plfdisdest} is stable. 
In Section \ref{secdiscrete} we shall consider a general discrete $\eta$.
But the balancing allocation will not be stable anymore.

\begin{remark}\label{4} \rm Proposition \ref{plfdisdest} can already be
found as Proposition 4.37 in \cite{OmidAli16}. There the authors used
a site-optimal version of a stable allocation while ours is point-optimal. 
\end{remark}

\section{Destination a discrete random measure}\label{secdiscrete}

In this section we deal with a random measure $\eta$ which
is discrete but not necessarily with a locally finite support (isolated atoms).
We need to introduce some notation.
If $A\subset\Omega\times \R^d$ is measurable 
then we identify $A$ with
the mapping $\omega\mapsto A(\omega):=\{x\in\R^d:(\omega,x)\in A\}$.
If $\xi$ is a random measure on $\R^d$, then we define for
each $\omega\in\Omega $ the restriction of $\xi(\omega)$
to $A(\omega)$ by
$\xi_A(\omega):=\int\I\{x\in \cdot, (\omega,x)\in A\}\xi(\omega,dx)$. Clearly 
$\xi_A$ is again a random measure.

\begin{proposition}\label{pdiscretedest}
Assume that $\xi$ and $\eta$
are jointly stationary and ergodic random measures on $\R^d$ such that
$\xi$ is diffuse, $\eta$ is discrete and $\infty > \lambda_\xi\ge \lambda_\eta>0$.
Then there exists a measurable $A\subset\Omega\times\R^d$ and
an allocation $\tau$ balancing $\xi_A$ and $\eta$. If
$\lambda_\xi= \lambda_\eta$, then $\tau$ balances $\xi$ and $\eta$.
\end{proposition}
{\em Proof:} Write $\eta = \sum_{n=1}^\infty \eta_n$,
where $\eta_1, \eta_2, \dots$ are the mutually singular discrete random measures
\begin{align*}
\eta_n = \sum_x  \I\{1/n \leq \eta\{x\}<1/(n-1)\}\eta(\{x\})\delta_x
\end{align*}
which all have non-accumulating (isolated) points.
(In the definition of $\eta_1$ we use the convention $1/0:=\infty$.)

Throughout this proof we consider the stable allocation $\tau^1$ with
appetite 1; see the definition preceding Lemma \ref{l11.6}.
Starting with $\xi^1:=\xi$ we define sequences of measurable sets $A_n\subset\Omega\times\R^d$,
$n\in\N$, and of stationary random measures $(\xi_n)_{n\ge 1}$ and $(\xi^n)_{n\ge 1}$
recursively, by setting for each $n\in\N$,
\begin{align*}
A_{n}&:=\{x\in\R^d:\tau^1(\xi^{n},\eta_{n},x)\ne\infty\},\\
\xi_{n}&:=\xi^{n}_{A_n},\\
\xi^{n+1}&:=\xi^{n}_{\R^d\setminus A_{n}}.
\end{align*}
Set $B_n:=A_1\cup\cdots \cup A_n$.
Using Proposition \ref{p1}
one can prove by induction that $\xi(A_{n+1}\cap B_n)=0$,
and $\lambda_{\xi_{n}}= \lambda_{\eta_{n}}$.
Note that $\sum^\infty_{n=1}\xi_n=\xi_A$, where $A:=\cup^\infty_{n=1}A_n$.

The calculation below shows that it is no restriction
of generality to assume that the sets $A_n$ are disjoint.
Therefore we can define an allocation $\tau$ by 
\begin{align}
\tau(x):=
\begin{cases}
\tau^1(\xi^n,\eta_n,x),&\text{if $x\in A_n$ for some $n\in\N$},\\
\infty, &\text{otherwise}.
\end{cases}
\end{align}
Then we obtain for each Borel set $C\subset\R^d$ that
\begin{align*}
\int \I\{\tau(x)\in C\}\xi_A(dx)
&=\sum^\infty_{n=1}\int \I\{\tau(x)\in C\}\xi_n(dx)\\
&=\sum^\infty_{n=1}\int \I\{\tau^1(\xi^n,\eta_n,x)\in C\}\xi_n(dx)\\
&=\sum^\infty_{n=1}\eta_n(C)=\eta(C).
\end{align*}
Therefore $\tau$ is balancing $\xi_A$ and $\eta$. 

Assume now that $\lambda_\xi=\lambda_\eta$.
Then we obtain for each Borel set $C\subset\R^d$ that
\begin{align*}
\BE \xi_A(C)=\sum^\infty_{n=1}\BE\xi_n(C)
=\sum^\infty_{n=1}\lambda_{\xi_n}\lambda_d(C)
=\sum^\infty_{n=1}\lambda_{\eta_n}\lambda_d(C)
=\lambda_\eta\lambda_d(C)=\lambda_\xi\lambda_d(C).
\end{align*} 
Therefore, $\BE \xi_A(C)=\BE \xi(C)$. Since $\xi_A\le\xi$,
this implies that $\xi_A=\xi$ $\BP$-a.s.
\qed

\section{Destination a Diffuse Random Measure}~\label{secdifdest} 

In this section we deal with a diffuse destination $\eta$ 
in the case when
there exists an auxiliary simple point process $\chi$. 
The key idea is to use the allocation from the isolated-atoms case 
(Proposition~\ref{plfdisdest}) 
to map both $\xi$ and $\eta$  to $\chi$ creating pairs of
 $\xi$-cells and $\eta$-cells of mass one 
 associated with each point of  $\chi$, 
and then to transport
the $\xi$-mass of the $\xi$-cells into the $\eta$-mass of the  $\eta$-cells 
by passing through Lebesgue measure on $[0,1]$.

\begin{proposition}\label{pdiffusedest}
Assume that $\xi$ and $\eta$ are 
jointly stationary and ergodic random measures on $\R^d$ such that
$0<\lambda_\xi=\lambda_\eta <\infty$. 
Assume further that $\xi$ and $\eta$ are both diffuse
and that there exists an auxiliary simple point process $\chi$ 
with finite intensity $\lambda_\chi$.
Then there exists an allocation $\tau$ balancing $\xi$ and $\eta$.
\end{proposition}

{\em Proof:} 
Note that an allocation balances the diffuse $\xi$ and $\eta$ 
if and only if it balances 
 $a\xi$ and $a\eta$ for any positive constant  $a$, 
 in particular for $a=\lambda_\chi/\lambda_\xi=\lambda_\chi/\lambda_\eta$.  
 So it is no restriction to 
 assume that the common intensity of $\xi$ and $\eta$ 
 is the same as that of $\chi$, that is, $\lambda_\xi=\lambda_\eta=\lambda_\chi$. 
 We can then apply 
Proposition \ref{plfdisdest}
to the pair $\xi$ and $\chi$ and to the pair $\eta$ and $\chi$.

For each point $s$ of $\chi$
let  $A_s$  and $B_s$, respectively, be the resulting
allocation cells of $\xi$ and $\eta$ that are mapped to $s$.
Fix $t\in \R^d$ and let $S_t$ be the point of $\chi$ such that $t \in A_{S_t}$. 
Let $\xi_{t}= \xi(\cdot \cap A_{S_t})$ be the restriction of 
$\xi$  to $A_{S_t}$ 
and let $\eta_{t}=\eta(\cdot \cap B_{S_t})$ be the restriction of 
$\eta$  to $B_{S_t}$. 
Since $\chi$ is a simple point process (with mass one at each of its points) both $\xi_t$ 
and $\eta_t$ 
are (random) probability measures.

Let $\phi$ be a measurable bijection from 
$\R^d$ to $\R$ 
such that $\phi^{-1}$ is also measurable.
Shift the origin $0$ to 
 $S_t$ to obtain (random) probability measures $\theta_{S_t}\xi_t$ 
and $\theta_{S_t}\eta_t$ that
are concentrated on the shifted cells
$A_{S_t}\!-\!S_t$ and $B_{S_t}\!-\!S_t$. 
Let $F_t$ and $G_t$ be the (random) distribution functions of 
$\phi$ under these probability measures, 
that is, for $x \in \R$
\begin{align*}
F_t(x) &= \theta_{S_t}\xi_t(\phi \leq x),\\
G_t(x) &=  \theta_{S_t}\eta_t(\phi \leq x).
\end{align*}
Note that $F_t$ 
is a continuous function since $\xi$  
does not have any atoms and since $\phi$ is 
a bijection. 
Thus

\vspace*{2mm}

\qquad\qquad the distribution of $F_t(\phi)$ under $\theta_{S_t}\xi_t$
is uniform on  $[0, 1]$.

\vspace*{2mm}

\noindent
With $G_t^{-1}$ the generalized inverse (quantile function) 
of $G_t$, this in turn implies that $G_t^{-1}(F_t(\phi))$
under $\theta_{S_t}\xi_t$ has the distribution function $G_t$. 
Finally, since $\phi$ is  a bijection this implies that

\vspace*{2mm}

\qquad\qquad  the distribution of $\phi^{-1}(G_t^{-1}(F_t(\phi)))$
under $\theta_{S_t}\xi_t$
is $\theta_{S_t}\eta_t$.

\vspace*{2mm}

\noindent
In other words, the mapping $x\mapsto \phi^{-1}(G_t^{-1}(F_t(\phi(x))))$ transports the 
measure $\theta_{S_t}\xi_t$ 
on $A_{S_t}\!-\!S_t$ into the measure $\theta_{S_t}\eta_t$ on $B_{S_t}\!-\!S_t$. 
Shifting back to the original origin, this means that the mapping 
$x\mapsto S_t + \phi^{-1}(G_t^{-1}(F_t(\phi(x-S_t)))$ transports the 
measure $\xi_t$ 
on $A_{S_t}$ into the measure $\eta_t$ on $B_{S_t}$. 
Thus,
 the allocation rule $\tau$ defined by 
$$
\tau(t) = S_t + \phi^{-1} (G_t^{-1}(F_t(\phi(t-S_t)))),\quad t\in\R^d,
$$
balances $\xi$ and $\eta$.   \qed

\begin{remark}\label{6}\rm
  The above construction, using a point process to split space into
  pairs of cells and then allocate mass through Lebesgue measure,
  dates back to an informal note from 2012. It was a part of a brief
  attempt of the authors to extend the Brownian motion results of
  \cite{LaMoeTho12} to higher dimensional random fields. The obvious
  question is when this `auxiliary' process does exist. We discussed
  this with several colleagues and the existence problem did become
  part of the
  PhD topic of Ali Khezeli, see Remark~\ref{8}.  In his thesis he 
  used a result on optimal transport to balance the finite masses of
  the pairs of allocation cells, under certain restrictions on the
  diffuse source.
    
\end{remark}

\section{Destination having Discrete and Diffuse Parts}~\label{secmixeddest}

In this section we finish the proof of Theorem~\ref{main} by 
dealing with the case when the destination $\eta$ 
contains both discrete and diffuse parts.

\begin{proposition}\label{pmixeddest}
Assume that $\xi$ and $\eta$ are 
jointly stationary and ergodic random measures on $\R^d$ such that
$\lambda_\xi=\lambda_\eta<\infty$. 
Assume further that $\xi$ is diffuse
and that $\eta$ is mixed, that is, 
$$
\eta = \eta^\emph{disc} + \eta^\emph{diff}
$$
where $\eta^\emph{disc}$ and $\eta^\emph{diff}$
are nonzero random measures that are discrete and 
diffuse respectively.
Then there exists an allocation $\tau$ balancing $\xi$ and $\eta$.
\end{proposition}

{\em Proof:} 
Note that the measures $\xi$, $\eta$, $\eta^\text{disc}$, $\eta^\text{diff}$ are 
jointly stationary and ergodic since $\eta^\text{disc}$ and $\eta^\text{diff}$ 
can be obtained from $\eta$ in an translation invariant way. 
Now apply Proposition~\ref{pdiscretedest} to $\xi$ and $\eta^\text{disc}$ using
 $\lambda_{\xi} \geq \lambda_{\eta^\text{disc}}$
 to obtain 
 a measurable $A\subset\Omega\times\R^d$ and
an allocation $\tau^\text{disc}$ balancing $\xi_A$ and $\eta^\text{disc}$. 
Then, 
with $\chi$ 
the auxiliary simple point process defined at \eqref{i3}, apply Proposition~\ref{pdiffusedest} 
 to $\xi_{A^c}$ and $\eta^\text{diff}$ using
 $\lambda_{\xi_{A^c}} = \lambda_{\eta^\text{diff}}$
 to obtain 
an allocation $\tau^\text{diff}$ balancing $\xi_{A^c}$ and $\eta^\text{diff}$. 
Finally, define an allocation $\tau$ by
$$
\tau(\omega, x) =
 \begin{cases}
\tau^\text{disc}(\omega, x),&\text{if $(\omega, x) \in A$},\\
\tau^\text{diff}(\omega, x),&\text{if $(\omega, x) \in A^c$}.
\end{cases}
$$
It is easy to see that $\tau$ balances $\xi$ and $\eta$. 
\qed

\section{Counterexample}\label{seccounter}

In this section we  show  that
the  auxiliary $\chi$ cannot simply be removed 
from Theorem~\ref{main}. There are 
diffuse $\xi$ and $\eta$ such that an allocation 
transporting $\xi$ to $\eta$ does not exist 
(without external randomisation).

Here is a specific example in two dimensions, $\R^2$. 
Let $N$
be a canonical stationary Poisson process on the $y$-axis with intensity $1$; 
canonical means that $N$ is the only source of randomness.
Let $\xi$ be formed by the one-dimensional 
Lebesgue measure 
on the lines parallel to the $x$-axis going through the points of 
$N$. Let $\eta=\lambda_2 =$ the Lebesgue measure on $\R^2$. 
The measures $\xi$ and $\eta$ are diffuse, 
jointly stationary and ergodic, 
and have the same intensity~$1$. 

Suppose there exists a balancing allocation  $\tau_N$ which
maps $\R^2$ to $\R^2$ for each fixed value of $N$ in such a way 
that the image measure of $\xi$ under $\tau_N$ is $\eta$. 
The Palm version of $N$ is $N^{\circ}=N+ \delta_0$ and the 
Palm version of $\xi$ is $\xi^{\circ}=\xi+\lambda_1$, where $\lambda_1$ 
denotes  one-dimensional  Lebesgue measure 
on the $x$-axis. 
According to \eqref{shift-coupling2224},
the existence of the allocation $\tau_N$ balancing 
$\xi$ and the two-dimensional 
Lebesgue measure $\eta=\lambda_2$
would yield the following shift-coupling of $\xi^{\circ}$ and $\xi$,
$$
\theta_{\tau_N(0)}\xi^{\circ}
 \,\overset{D}{=}\, \xi .
$$
With $T_N$  the $y$-axis coordinate of $\tau_N(0)$, 
 this implies that
$$ \theta_{T_N}N^{\circ}
 \,\overset{D}{=}\, N. $$
But, according  to  \cite{HP05}, such a $T_N$ does not exist 
when the only sorce of randomness is $N$.

\begin{remark}\label{8}\rm  When reading a preliminary version of this paper, 
Ali Khezeli 
pointed out to the authors the following interesting problem,
formulated in his PhD-thesis from 2016 (in Persian). Suppose that $\xi$ is a diffuse random
measure with no invariant directions. This means that there is (almost surely)
no vector $x\ne 0$ such $\xi=\theta_{tx}\xi$ for all $t\in\R$. Does $\xi$ have
a stationary point process factor? A positive answer would bring us much
closer to a complete characterisation of the existence of balancing
factor allocations (for a diffuse source). 
Note that our counterexample has an invariant direction.
\end{remark}

\section{Remarks}\label{secrem}

\begin{remark}\label{ergod}\rm
The assumption of ergodicity has been made
for simplicity and can be relaxed. The assumption $\lambda_\xi= \lambda_\eta$
has then to be replaced by
$$
\BE[\xi[0,1]\mid \mathcal{I}]=\BE[\eta[0,1]\mid \mathcal{I}],\quad \BP\text{-a.e.}
$$
We refer to \cite{LaTho09, Thor96} for more detail on this point.
\end{remark}

\begin{remark}\label{ali-omid}\rm
A natural question (asked in \cite{OmidAli16}
for instance) is whether there exists a balancing
allocation  factor
  $\tau$ 
if the source $\xi$ is Lebesgue measure or, more generally,
absolutely continuous with respect to Lebesgue measure.
We have proved the answer to be positive whenever the destination
$\eta$  is {\em not} diffuse (not purely non-atomic), and also 
when $\eta$  is diffuse if we assume
that there exists an auxiliary point process factor~$\chi$.

But this assumption is not necessary, 
a balancing
allocation  factor $\tau$ {\em can} exist {\em even} when 
$\eta$ is diffuse and 
{\em no} auxiliary point process factor $\chi$ exists.
An example of this is obtained 
by swapping 
source and destination 
in Section~\ref{seccounter}: take $\xi=\lambda_2 =$ the 
Lebesgue measure on $\R^2$ 
and let
 $\eta$ be formed by the one-dimensional 
Lebesgue measure 
on the lines parallel to the $x$-axis going through the points of 
$N$ where $N$ is a one-dimensional Poisson process 
on the $y$-axis. 
If $N$ is the only source of randomness then, according to
Section~\ref{seccounter}, there exists {\em no}
 auxiliary $\chi$. However,  
 there exists a balancing allocation factor $\tau$: for example, 
 take
 $\tau(x, y) = (x, \tau_1(y))$ where $\tau_1$ is an 
 allocation balancing Lebesgue measure on the line (the $y$-axis)
 and the Poisson process~$N$.  
 
 Thus, the existence of an auxiliary point process 
 is not a complete characterisation of the existence of 
 a balancing allocation factor when the source is diffuse. 
\end{remark}

\begin{remark}\label{auxfactor}\rm
If the source $\xi$ is not diffuse, then the question asked
in this paper is in most cases not very meaningful.
If, for instance, the destination $\eta$ is diffuse,
then a balancing allocation cannot exist. But even otherwise such
allocations can only exist in special cases, for instance,
if both $\xi$ and $\eta$ are simple point processes.
\end{remark}

\begin{remark}\label{alternativ}\rm
The following slight modification of the construction in Section~\ref{secdifdest}
 (where the destination is diffuse)
can be used to obtain allocations in  the
 cases treated in Section~\ref{secdiscrete} and Section~\ref{secmixeddest}.\ 
 Let $\chi$ be the simple point process defined at \eqref{i3}.\ 
 In the proof of Theorem~\ref{pdiffusedest} remove the first paragraph,
 replace the allocation cells 
 $B_s$ of $\eta$ by the Voronoi cells $V_s$ of $\chi$,
 and then modify $\chi$ by letting it have mass $\eta(V_s)$ at each point $s$. 
After this modification we have
$\lambda_\xi=\lambda_\eta=\lambda_\chi$ 
and  can apply 
Proposition \ref{plfdisdest}
to the pair $\xi$ and $\chi$.
Let  $A_s$  be the 
allocation cell of $\xi$ that is mapped to $s$. 
 This yields  $F_t$ and $G_t$ that are  
 cumulative  mass 
 functions of measures 
 that need not have mass $1$ but only have the same finite mass, 
 $F_t(\infty) = \xi(A_{S_t})= \eta(V_{S_t})=G_t(\infty)$. 
Replace first the word {\em distribution function} by {\em cumulative  mass 
 function}
and then the word {\em distribution} by {\em image messure}.
The rest of the proof now goes through as it stands.
 
This method thus yields allocations in all the
three cases where $\eta$ does not consist of isolated atoms.
We have chosen here to treat each of those cases separately
because treating discrete measures as in Section~\ref{secdiscrete} is  more explicit
then this alternative method.
\end{remark}

\begin{remark}\label{history}\rm 
Here are some further historical notes. Allocations finding extra heads in 
coin tosses 
on the $d$ dimensional grid and extra points in 
the $d$ dimensional Poisson process 
were constructed in \cite{HL01, HP05, HoHolPe06, CPPeresR10}
by Liggett, Holroyd, Hoffman, Peres et.al. 
More generally, the allocations 
transporting Lebesgue measure to the points of 
stationary ergodic finite-intensity point processes 
produce the Palm versions of the processes. 
The construction in \cite{HoHolPe06}
 involved a Gale-Shapley algorithm resulting in a `stable' allocation
while the construction in \cite{CPPeresR10} 
used a gravitational force field to obtain an `economical' allocation. 
In \cite{Hues16},  
unique optimal allocations between jointly stationary random measures
on geodesic manifolds were constructed, assuming the (Palm) average cost to be finite
and the source to be absolutely continuous.
Stable transports between general (jointly stationary)
random measures $\xi$ and $\eta$ on $\R^d$ were
constructed and studied in \cite{OmidAli16}. If $\xi$ is diffuse
and $\eta$ is a point process, then these
transports boil down to allocations, 
see Remarks \ref{4}, \ref{6} and \ref{8}. 
\end{remark}

\bigskip
\noindent
{\bf Acknowledgments:} 
 We wish to thank Ali Khezeli for helpful discussions 
 and an anonymous referee for  a clear-sighted and constructive 
 report resulting in a thorough rewriting of parts of the paper.
 This work was supported by the German Research Foundation
(DFG) through Grant No. 
LA 965/11-1 as part of the DFG priority programme
``Random Geometric Systems''.

\end{document}